%% file: main.tex
\documentclass[10pt,twoside,a4]{amsart}
\usepackage{amsmath,amsthm,latexsym,amssymb,pgf,url,tabularx} 

\usepackage{tikz,multicol,adjustbox,graphicx} 
\usetikzlibrary{arrows,decorations.pathmorphing,backgrounds,positioning,fit,petri,calc}

\newcommand{\old}[1]{{}}
\usepackage[section]{placeins}

\usepackage[font=small,labelfont=bf]{caption}

\newtheorem{theo}{Theorem}

\newtheorem{la}{Lemma}

\title{Proximity in Triangulations and Quadrangulations} 
\author{\'Eva Czabarka, Peter Dankelmann, Trevor Olsen, L\'aszl\'o A. Sz\'ekely }
\address{\'Eva Czabarka and L\'aszl\'o A. Sz\'ekely\\ Department of Mathematics \\ University of South Carolina \\ Columbia, SC 29208 \\ USA, and Visiting Professor\\ Department of Pure and Applied Mathematics \\ University of Johannesburg \\ P.O. Box 524\\ Auckland Park, Johannesburg 2006 \\ South Africa}
\email{\{czabarka,szekely\}@math.sc.edu }
\thanks{The second author was supported in part by 
the National Research Foundation of South Africa, grant number 118521,  
the third and fourth authors were supported by the NSF DMS, grant number  1600811.}
\address{Peter Dankelmann\\ Department of Pure and Applied Mathematics \\ University of Johannesburg \\ P.O. Box 524\\ Auckland Park, Johannesburg 2006 \\ South Africa}
\email{pdankelmann@uj.ac.za}
\address{Trevor Olsen\\ Department of Mathematics \\ University of South Carolina \\ Columbia, SC 29208 \\ USA}
\email{tvolsen@email.sc.edu }


\begin{document}

\begin{abstract}

Let $ G $ be a connected graph. If $\overline{\sigma}(v)$ denotes the arithmetic mean of the distances from $v$ to all other vertices of
$G$, then the proximity, $\pi(G)$, of $G$ is defined as the 
smallest value of $\overline{\sigma}(v)$ over all vertices $v$ of $G$. 
 We give %
 upper bounds for the proximity of simple triangulations and quadrangulations of given order and connectivity. We also construct
 simple triangulations and quadrangulations of given order and connectivity that match the upper bounds asymptotically and are likely
 optimal.

\end{abstract}

\maketitle

\input{content}
\pagebreak

\end{document}

%% file: content.tex
\section{Introduction}
Let $G$ be a connected graph with vertex set $V(G)$, and let $v$ be a vertex of $G$. 
The \emph{total distance} $\sigma(v)$ and the \emph{average
distance} $\overline{\sigma}(v)$ of $v$ is  
defined as the sum and the average, respectively, of the distances from $v$ to all 
other vertices. Bounds on $\sigma(v)$ were obtained, for example, in \cite{BarEntSze1997}
\cite{EntJacSny1976} and \cite{Zel1968}. 
Of particular interest are the minimum
value and the maximum value over all $v\in V(G)$ of $\overline{\sigma}(v)$ in a graph $G$,  
usually referred to as the \emph{proximity} $\pi(G)$ and the \emph{remoteness} $\rho(G)$, 
respectively, of $G$. A closely related graph invariant is the \emph{minimum status} of 
$G$, defined as the minimum value over all $v\in V(G)$ of $\sigma(v)$. 

It was shown by Zelinka \cite{Zel1968} and, independently, by 
Aouchiche and Hansen \cite{AouHan2011} that the proximity and remoteness of a 
connected graph of order $n$ are at most
approximately $\frac{n}{4}$ and $\frac{n}{2}$, respectively. For graphs of given
minimum degree $\delta$ these bounds were improved in \cite{Dan2015} by a factor of 
about $\frac{3}{\delta+1}$. The difference between remoteness and proximity in
a connected graph of order $n$ was shown to be at most about $\frac{n}{4}$ 
(see \cite{AouHan2011}), and this was improved by a factor of about $\frac{3}{\delta+1}$
in \cite{Dan2016}.  For more recent results on proximity and remoteness see, for example, 
\cite{MaWuZha2012} and \cite{WuZha2012}.

This paper is concerned with bounds on the proximity of triangulations, i.e., maximal planar
graphs, and quadrangulations, i.e.,  maximal bipartite planar graphs. 
Several bounds on distance measures in maximal planar graphs are known, for example for 
radius \cite{AliDanMuk2012}, average eccentricity \cite{AliDanMorMukSwaVet2018},
Wiener index \cite{CheCol2019, CzaDanOlsSze2019, GhoGyoPauSalZam2019, GyoPauXia2020}. (The 
Wiener index of a graph is the sum of the distances between all unordered pairs of graphs. 
The radius is the smallest of the eccentricities of the vertices of a graph, where the
eccentricity of a vertex $v$ is the distance to a vertex farthest from $v$.)
Sharp upper bounds on the remoteness of maximum planar graphs were given in 
\cite{CheCol2019} and  \cite{CzaDanOlsSze2019},  in the latter one sharp upper bounds on the remoteness of maximum bipartite planar graphs as well. However, no bound on the 
proximity of maximum planar graphs appears to be known. The aim of this 
paper is to fill this gap and give upper bounds
on proximity of simple triangulations and quadrangulations   of given order and connectivity.
 Our matching constructions make these bounds tight within 
an additive constant. We conjecture that these constructions are of maximum proximity in their respective classes.
It is a tedious routine calculation to determine the minimum status (i.e. the conjectured  maximum proximity in the respective class)
in the constructions, we just provide the results in formulae (\ref{tpoly}), (\ref{t4cpoly}), (\ref{t5cpoly}), (\ref{qpoly}), and    (\ref{q3cpoly}). Remarkably, all of the structures we found to maximize the proximity are identical to the structures we conjecture to maximize the Wiener Index in \cite{CzaDanOlsSze2019}. Similarly, the structures match the graphs which maximize the remoteness, outside of 5-connected triangulations on $ 5k+3 $ vertices.

The notation we use in this paper is as follows. If $G$ is a graph, then we denote
its vertex set by $V(G)$, and by $n(G)$ we mean the {\em order}, defined as $|V(G)|$. 
The {\em eccentricity} $e(v)$ of a vertex $v$ is the distance to a vertex farthest 
from $v$, i.e., $e(v) = \max_{u\in V(G)} d_G(v,u)$. The largest and the smallest 
of the eccentricities of the vertices of $G$ are the {\em diameter} and the {\em radius}
of $G$, respectively.  
The {\em neighbourhood} of a vertex
$v$ of $G$ is the set of vertices adjacent to $v$, it is denoted by $N_G(v)$, and
the cardinality $|N_G(v)|$ is the {\em degree} of $v$, which we denote by ${\rm deg}_G(v)$.   
If $i$ is an integer with $0 \leq i \leq e(v)$, then $N_i(v)$ denotes the set of all
vertices at distance exactly $i$ from $v$, and we write $n_i(v)$ for $|N_i(v)|$. 
If there is no danger of confusion, we often omit the subscript $G$ or the argument
$G$ or $v$.  If $A,B \subseteq V(G)$, then $m(A,B)$ denotes the number of 
edges that join a vertex in $A$ to a vertex in $B$, and $G[A]$ denotes 
the subgraph of $G$ induced by $A$.

We say that  a graph $G$ is $k$-\emph{connected}, 
$k\in \mathbb{N}$, if deleting fewer than $k$ vertices from $G$ always
leaves a connected graph.

\section{Upper bounds on proximity of triangulations and quadrangulations}

In this section we present bounds on the proximity of 
$k$-connected triangulations for $k\in \{3,4,5\}$  and 
$k$-connected quadrangulations for $k\in \{2,3\}$. All bounds
are sharp apart from an additive constant.

Our strategy is as follows: If $G$ is a $k$-connected triangulation 
or quadrangulation, then we choose a central vertex $v$ and 
derive certain properties of the sequence 
$n_0(v), n_1(v), n_2(v),\ldots,n_r(v)$. These properties will be 
used to obtain a bound on the average distance of $v$, which in turn is a 
bound on the proximity of $G$. 
The following lemmata were proved in \cite{AliDanMuk2012} (see the proofs of
Theorems 2.7, 2.9 and 2.10 there). 

\begin{la} {\rm  \cite{AliDanMuk2012} } 
\label{la:bound-on-ni-for-3-connected}
Let $G$ be a $3$-connected plane graph of radius $r$ with maximal face length $\ell$. 
If $v$ is a central vertex of $G$, then \\
(i) $n_i(v) \geq 3$ for 
 $i\in \{1,2,\ldots,\lfloor \frac{\ell}{2} \rfloor\} \cup 
   \{ r-\lfloor \frac{\ell}{2} \rfloor,  r-\lfloor \frac{\ell}{2} \rfloor +1,\ldots,r-1\}$, \\
(ii) $n_i(v) \geq 4$ for 
 $i\in \{\lfloor \frac{\ell}{2} \rfloor+1, \lfloor \frac{\ell}{2} \rfloor+2, \ldots, \ell \} \cup 
   \{ r-\ell, r-\ell+1,\ldots, r-\lfloor \frac{\ell}{2} \rfloor-1\}$, \\   
(iii) $n_i(v) \geq 6$ for
$i \in \{\ell+1, \ell+2,\ldots, r-\ell-1\}$.    
\end{la}

\begin{la} {\rm  \cite{AliDanMuk2012} } 
\label{la:bound-on-ni-for-4-connected}
Let $G$ be a $4$-connected plane graph of radius $r$ with maximal face length $\ell$. 
If $v$ is a central vertex of $G$, then \\
(i) $n_i(v) \geq 4$ for 
 $i\in \{1,2,\ldots,\ell\} \cup 
   \{ r-\ell, r-\ell+1,\ldots,r-1\}$, \\
(ii) $n_i(v) \geq 6$ for 
 $i\in \{\ell+1, \ell+2, \ldots, \lfloor \frac{3\ell}{2} \rfloor \} \cup 
   \{ r-\lfloor \frac{3\ell}{2} \rfloor,  r-\lfloor \frac{3\ell}{2} \rfloor +1,
          \ldots, r-\ell-1 \} $, \\   
(iii) $n_i(v) \geq 8$ for
$i \in \{ \lfloor \frac{3\ell}{2} \rfloor+1, \lfloor \frac{3\ell}{2} \rfloor+2, 
   \ldots, r-\lfloor \frac{3\ell}{2} \rfloor-1\}$.    
\end{la}

\begin{la} {\rm  \cite{AliDanMuk2012} } 
\label{la:bound-on-ni-for-5-connected}
Let $G$ be a $5$-connected plane graph of radius $r$ with maximal face length $\ell$. 
If $v$ is a central vertex of $G$, then \\
(i) $n_i(v) \geq 5$ for 
 $i\in \{1,2,\ldots,\ell\} \cup 
   \{ r-\ell, r-\ell+1,\ldots,r-1\}$, \\
(ii) $n_i(v) \geq 6$ for 
 $i\in \{\ell+1, \ell+2, \ldots, \lfloor \frac{3\ell}{2} \rfloor \} \cup 
   \{ r-\lfloor \frac{3\ell}{2} \rfloor,  r-\lfloor \frac{3\ell}{2} \rfloor +1,
          \ldots, r-\ell-1 \} $, \\   
(iii) $n_i(v) \geq 8$ for
$i \in \{ \lfloor \frac{3\ell}{2} \rfloor+1, \lfloor \frac{3\ell}{2} \rfloor+2, 
   \ldots, 2\ell\} \cup \{r-2\ell, r-2\ell+1, \ldots, r-\lfloor \frac{3\ell}{2} \rfloor-1\}$. \\
(iii) $n_i(v) \geq 10$ for
$i \in \{ 2\ell+1, 2\ell+2, 
   \ldots, r-2\ell-1\}$.       
\end{la}

We also need a corresponding result for $2$-connected quadrangulations. 
The following two lemmata follow the approach in \cite{AliDanMuk2012}. 
Given a fixed vertex $v$, we say that a vertex  in $N_i(v)$ is {\em active} 
if it has a neighbour in $N_{i+1}(v)$. The set of active vertices in $N_i(v)$ 
is denoted by $A_i(v)$.

\begin{la} \label{la:active-vertices-in-quadrangulation}
Let $G$ be a quadrangulation, $v$ a vertex of $G$ and $i\in \mathbb{N}$
with $1 \leq i \leq e(v)-1$. For every active vertex $w \in N_i(v)$ there exists
another active vertex $w' \in N_i(v)$ such that $w$ and $w'$ share a 
face of $G$. 
\end{la}

\begin{proof}
Let $u$ be an arbitrary vertex in $A_i$. Since u is active, it has neighbours 
in $N_{i-1}$ and in $N_{i+1}$. Number the neighbours of $u$ as $x_0, x_1, \ldots, x_t$ 
such that the edges $ux_j$ appear in clockwise order, $x_0$ is in $N_{i-1}$ and, 
say, $x_k$ is in $N_{i+1}$. Denote the face containing $u$, $x_j$, $x_{j+1}$
and a fourth vertex by $f_j$ for $j = 0, 1, \ldots, t$, where subscripts 
are taken modulo $t+1$, and let
$y_j$ be the vertex on $f_j$ not equal to $u$, $x_j$ and $x_{j+1}$.

Consider the $(x_0,x_k)$-walk $W:x_0, y_0,x_1, y_1,\ldots,x_{k-1},y_{k-1}x_k$.
Since $W$ joins a vertex in $N_{i-1}$ to a vertex in $N_{i+1}$, it contains 
a vertex in $N_i$. Since $G$ is bipartite, none of the neighbours of $u$ is in 
$N_i$, so there exists a vertex $y_j$ which is in $N_i$.  We may assume that
$y_j$ is the last such vertex. Then $y_j$ has a neighbour in $N_{i+1}$ and is
thus active. Since $u$ and $y_j$ share a face, Lemma \ref{la:active-vertices-in-quadrangulation} follows.  
\end{proof}

\begin{la} \label{la:bound-on-ni-for-2-connected-quadr}
Let $G$ be a quadrangulation of order $n$ and radius $r$. 
If $v$ is a central vertex of $G$, then \\
(i) $n_i \geq 2$ for $i\in \{1,2,r-2,r-1\}$, and  \\
(ii) $n_i \geq 4$ for $i\in \{3,4,\ldots,r-3\}$. 
\end{la} 

\begin{proof}
Let $v$ be a central vertex of the quadrangulation $G$ and let 
$i\in \{1,2,\ldots,r-1\}$, where $r$ is the radius of $G$. \\
(i) Clearly, $N_i$ contains an active vertex. 
It follows from Lemma \ref{la:active-vertices-in-quadrangulation}
that every active vertex in $N_i$ shares a face with some other 
active vertex in $N_i$. Hence $N_i$ contains at least two active
vertices, and so (i) follows. \\
(ii) Suppose to the contrary that (ii) does not hold. Then there 
exists $i \in \{3,4,\ldots,r-3\}$
such that $n_i \leq 3$. 

Denote the set of active vertices in $N_i$ by $A_i$.  
Since by Lemma \ref{la:active-vertices-in-quadrangulation} every vertex
in $A_i$ shares a face with some other vertex in $A_i$, and since
$|A_i| \leq |N_i| =3$, 
there exist a vertex $z_i \in A_i$ that shares a face in $G$ with every 
other vertices of $A_i$. 
Since all faces of $G$ have length $4$, it follows that $d_G(z_i,y_i)\leq 2$ for
all $y_i \in A_i$. Let $z_3\in N_3$ be a vertex on a shortest $(v,z_i)$-path in $G$,
so that $d_G(z_3,z_i) = i-3$. 

We now bound $d(z_3,x)$ for all $x\in V(G)$. 
First let $x \in \bigcup_{j=0}^{r-4} N_i$. Then 
$d_G(z_3,x) \leq d_G(z_3,v) + d_G(v,x) \leq 3+(r-4) =r-1$. 
Now let $x \in \bigcup_{j=r-3}^r N_i$. Let $x_i \in N_i$ be a vertex
on a shortest $(v,x)$-path in $G$. Then $x_i\in A_i$, and so $d_G(z_i,x_i) \leq 2$. 
Hence 
$d_G(z_3,x) \leq d_G(z_3,z_i) + d_G(z_i,x_i) + d_G(x_i,x) 
        \leq (i-3) + 2 + (r-i) = r-1$. 
So $d_G(z_3,x) \leq r-1$ in all cases, a contradiction to $r$ being the radius
of $G$. Hence our assumption $n_i \leq 3$ is false, and (ii) follows.         
\end{proof}

For the proofs below we define the function $F$ which assigns to a
finite sequence $X=(x_0,x_1,\ldots,x_k)$ of integers the value 
$F(X)=\sum_{i=0}^k ix_i$. So if $v$ is a vertex of eccentricity $r$ 
in a connected graph $G$, then 
$\sigma(v) = \sum_{i=0}^r in_i(v) = F(n_0,n_1,\ldots,n_r)$.

\begin{theo} \label{prop:bound-on-proximity}
Let $G$ be a planar graph of order $n$ and $v$ a 
central vertex of $G$. \\
(a) If $G$ is a triangulation, then
\[ \pi(G) \leq \frac{n+19}{12} + \frac{25}{3(n-1)}. \]
(b) If $G$ is a $4$-connected triangulation, then
\[ \pi(G) \leq \frac{n+35}{16} + \frac{91}{4(n-1)}. \]
(c) If $G$ is a $5$-connected triangulation, then
\[ \pi(G) \leq \frac{n+57}{20} + \frac{393}{10(n-1)}. \]
(d) If $G$ is a quadrangulation, then
\[ \pi(G) \leq \frac{n+11}{8} + \frac{9}{2(n-1)}. \]
(e) If $G$ is a $3$-connected quadrangulation, then
\[ \pi(G) \leq \frac{n+25}{12} + \frac{169}{12(n-1)}. \]
\end{theo}

\begin{proof}
We only prove part (a) of the theorem; the proofs of (b)-(e) are analogous. 
Let $G$ be a simple triangulation, let $r$ be its radius and let $v$ be 
a central vertex of $G$.  
It suffices to prove that 
\[ \sigma(v) \leq \frac{1}{12}(n^2+18n+81). \]
Let $n_i:= n_i(v)$ for $i=0,1,\ldots,r$. Then 
$\sigma(v) = \sum_{i=1}^r in_i = F(n_0,n_1,\ldots,n_r)$. 
By Lemma \ref{la:bound-on-ni-for-3-connected}, we have $n_1, n_{r-1}\geq 3$, 
$n_2,n_3, n_{r-3},n_{r-2}\geq 4$, and $n_i\geq 6$ for $i=4,5,\ldots,r-4$. 
Moreover, we have $n_0=1$, $n_r\geq 1$, and $\sum_{i=0}^r n_i=n$. 

Now assume that for given $n$ the integers $r', n_0',n_1'\ldots,n_{r'}'$ 
are chosen to maximise $F(n_0',n_1',\ldots,n_{r'}')$ subject to the above 
conditions (with $n_i$ and $r$ replaced by $n_i'$ and $r'$, respectively). Then clearly  
$n_0'=1$, $n_1'=3$, $n_2'=n_3'=4$, $n_i'=6$ for $i=4,5,\ldots,r'-4$,
and $n_{r'-3}'=n_{r'-2}'=4$, $n_{r'-1}'=3$. 

Consider the sequence $X^*=(n_0',n_1',\ldots,n_{r'-1}',1)$. The sum of the entries
of $X^*$ is $n-n_{r'}+1$. Simple calculations show that $r'=\frac{1}{6}(n-n_{r'}'+19)$ 
and $F(X^*) = \frac{1}{2}r'(n-n_{r'}'+1)$. Hence 
\[ F(n_0',n_1',\ldots,n_{r'}') = F(X^*) + r'(n_{r'}'-1) = \frac{1}{2}r'(n+n_{r'}'-1). \] 
Substituting $r'=  \frac{1}{6}(n-n_{r'}'+19)$ yields, after simplification, 
$F(n_0',n_1',\ldots,n'_{r'}) = \frac{1}{12}(n^2 + 18n -(n_{r'}')^2 + 20n_{r'}'-19)$. 
Since the function $-x^2+20x$ attains its maximum $100$ for $x=10$, we have 
$-(n_{r'}')^2 + 20n_{r'}' \leq 100$, and so we get 
$F(n_0',n_1',\ldots,n'_{r'}) \leq \frac{1}{12}(n^2 + 18n +81)$ and thus
\[ \sigma(v)  = F(n_0,n_1,\ldots,n_r) \leq F(n_0', n_1',\ldots, n_{r'}')
       \leq \frac{1}{12}(n^2 + 18n +81), \]
as desired.                      
\end{proof}

The bounds in Theorem \ref{prop:bound-on-proximity} appear not to be sharp.
The graphs constructed in the following section show that the bounds are
sharp up to an additive constant.

\section{Computational Results}

This section includes figures of the extremal structures which minimize the status given certain connectivity. Additionally, formulae and tables summarize these minimized values. This paper heavily utilized a software package called Plantri, we are grateful for their hard work and dedication to the exploration of planar graphs. 
For each category of problem (triangulations, 4-connected triangulations, 5-connected triangulations, quadrangulations and 3-connected quadrangulations) there is a function
 (see  formulae (\ref{tpoly}), (\ref{t4cpoly}), (\ref{t5cpoly}), (\ref{qpoly}),    and    (\ref{q3cpoly})), which states the minimum status of the structures, along with a table, which summarizes the largest minimum status found for a given order in that category and a ``Count'', summarizing how many graphs attain the optimal value. (Recall that the minimum status  is defined by $(n-1){\pi (G)}$.) We use the minimum status rather than the proximity in the tables to remain in the domain of integers. If any Table entry displays a dash, there exists no graphs in that category on the given number of vertices. We searched the same number of isomorphism classes as \cite{BowFis1967}, \cite{BriGGMTW2005}, \cite{BriMck2005}, \cite{Lut2008}, \cite{SchSchSte2018}, verifying that our values are in fact optimal. In each figure below, red edges represent the repeating pattern and  the black node marks one of the vertices which minimizes the status in that graph and likely maximizes the proximity within the category defined by order and connectivity. Due to fact that few congruence classes of a large modulus fall into the range, where we can do exhaustive calculations,
 finding a repeating pattern was difficult by brute force calculations. Once a pattern was established, extensive sampling was conducted in order to test if a better structure could be found, but no such structure ever arose. We conjecture that the bounds presented in this section are optimal.

\begin{figure}[htbp] 
	\centering
		\begin{tikzpicture}
		[scale=0.8,xscale=1,yscale=0.9,rotate=180,inner sep=1mm, %
		vertex/.style={circle,thick,draw}, %
		thickedge/.style={line width=2pt}] %
		\node[vertex] (a1) at (0,4) [fill=white] {};
		\node[vertex] (a2) at (2,4) [fill=white] {};
		\node[vertex,red] (a3) at (4,4) [fill=white] {};
		\node[vertex,red] (a4) at (6,4) [fill=white] {};
		\node[vertex] (a5) at (8,4) [fill=white] {};
		\node[vertex] (a6) at (10,4) [fill=white] {};
		
		\node[vertex] (b1) at (0,0) [fill=white] {};
		\node[vertex] (b2) at (2,0) [fill=white] {};
		\node[vertex,red] (b3) at (4,0) [fill=white] {};
		\node[vertex,red] (b4) at (6,0) [fill=white] {};
		\node[vertex] (b5) at (8,0) [fill=white] {};
		\node[vertex] (b6) at (10,0) [fill=white] {};
		
		\node[vertex] (c1) at (0.5,2) [fill=white] {};
		\node[vertex] (c2) at (2.5,2) [fill=white] {};
		\node[vertex,red] (c3) at (4.5,2) [fill=black] {};
		\node[vertex,red] (c4) at (6.5,2) [fill=white] {};
		\node[vertex] (c5) at (8.5,2) [fill=white] {};
		\node[vertex] (c6) at (10.5,2) [fill=white] {};
		
		\draw[very thick] (a1)--(a2)--(a3) (a5)--(a6);  
		\draw[very thick] (b1)--(b2)--(b3) (b5)--(b6);     
		\draw[very thick] (c1)--(c2)--(c3) (c5)--(c6); 
		\draw[very thick] (a1)--(b1)--(c1)--(a1);
		\draw[very thick] (a2)--(b2)--(c2)--(a2);
		\draw[very thick,red] (a3)--(b3)--(c3)--(a3);
		\draw[very thick,red] (a4)--(b4)--(c4)--(a4);
		\draw[very thick] (a5)--(b5)--(c5)--(a5);
		\draw[very thick] (a6)--(b6)--(c6)--(a6);
		
		\draw[very thick] (c1) -- (a2);  
		\draw[very thick] (c1) -- (b2);  		
		\draw[very thick] (c2) -- (a3);  
		\draw[very thick] (c2) -- (b3);
		\draw[very thick,red] (c3) -- (a4);  
		\draw[very thick,red] (c3) -- (b4);
		\draw[very thick,red] (c4) -- (a5);  
		\draw[very thick,red] (c4) -- (b5);
		
		\draw[very thick] (a1) -- (b2);  
		\draw[very thick] (a2) -- (b3);
		\draw[very thick,red] (a3) -- (b4);
		\draw[very thick,red] (a4) -- (b5);
		
		\draw[very thick,red] (a3) -- (a4);
		\draw[very thick,red] (b3) -- (b4);
		\draw[very thick,red] (c3) -- (c4);
		\draw[very thick,red] (c3) -- (c4);
		
		\draw[very thick] (c5) -- (a6);
		\draw[very thick] (c5) -- (b6);
		\draw[very thick] (a5) -- (b6);
		
		\draw[very thick,red] (b4) -- (b5);
		\draw[very thick,red] (a4) -- (a5);
		\draw[very thick,red] (c4) -- (c5);

		\end{tikzpicture}
		\caption{A triangulation $ T_n $ on $ n = 6k $ vertices which is conjectured to maximize the proximity among triangulations of this order.}
		\label{fig:t0}
	
\end{figure}
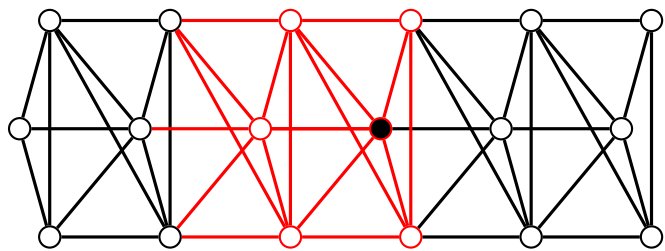

\begin{figure}[htbp] 
	\centering
		\begin{tikzpicture}
		[scale=0.8,xscale=1,yscale=0.9,rotate=180,inner sep=1mm, %
		vertex/.style={circle,thick,draw}, %
		thickedge/.style={line width=2pt}] %
		\node[vertex] (a1) at (0,4) [fill=white] {};
		\node[vertex] (a2) at (2,4) [fill=white] {};
		\node[vertex,red] (a3) at (4,4) [fill=white] {};
		\node[vertex,red] (a4) at (6,4) [fill=white] {};
		\node[vertex] (a5) at (8,4) [fill=white] {};
		\node[vertex] (a6) at (10,4) [fill=white] {};
		
		\node[vertex] (b1) at (0,0) [fill=white] {};
		\node[vertex] (b2) at (2,0) [fill=white] {};
		\node[vertex,red] (b3) at (4,0) [fill=white] {};
		\node[vertex,red] (b4) at (6,0) [fill=white] {};
		\node[vertex] (b5) at (8,0) [fill=white] {};
		\node[vertex] (b6) at (10,0) [fill=white] {};
		
		\node[vertex] (c1) at (0.5,2) [fill=white] {};
		\node[vertex] (c2) at (2.5,2) [fill=white] {};
		\node[vertex,red] (c3) at (4.5,2) [fill=black] {};
		\node[vertex,red] (c4) at (6.5,2) [fill=white] {};
		\node[vertex] (c5) at (8.5,2) [fill=white] {};
		\node[vertex] (c6) at (10.5,2) [fill=white] {};
		\node[vertex] (c0) at (-1.5,2) [fill=white] {};
		
		\draw[very thick] (a1)--(a2)--(a3) (a5)--(a6);  
		\draw[very thick] (b1)--(b2)--(b3) (b5)--(b6);     
		\draw[very thick] (c0)--(c1)--(c2)--(c3) (c5)--(c6); 
		\draw[very thick] (a1)--(b1)--(c1)--(a1);
		\draw[very thick] (a2)--(b2)--(c2)--(a2);
		\draw[very thick,red] (a3)--(b3)--(c3)--(a3);
		\draw[very thick,red] (a4)--(b4)--(c4)--(a4);
		\draw[very thick] (a5)--(b5)--(c5)--(a5);
		\draw[very thick] (a6)--(b6)--(c6)--(a6);
		
		\draw[very thick] (c1) -- (a2);  
		\draw[very thick] (c1) -- (b2);  		
		\draw[very thick] (c2) -- (a3);  
		\draw[very thick] (c2) -- (b3);
		\draw[very thick,red] (c3) -- (a4);  
		\draw[very thick,red] (c3) -- (b4);
		\draw[very thick,red] (c4) -- (a5);  
		\draw[very thick,red] (c4) -- (b5);
		
		\draw[very thick] (a1) -- (b2);  
		\draw[very thick] (a2) -- (b3);
		\draw[very thick,red] (a3) -- (b4);
		\draw[very thick,red] (a4) -- (b5);
		
		\draw[very thick,red] (a3) -- (a4);
		\draw[very thick,red] (b3) -- (b4);
		\draw[very thick,red] (c3) -- (c4);
		\draw[very thick,red] (c3) -- (c4);
		
		\draw[very thick] (c5) -- (a6);
		\draw[very thick] (c5) -- (b6);
		\draw[very thick] (a5) -- (b6);
		
		\draw[very thick] (c0) -- (a1);
		\draw[very thick] (c0) -- (b1);
		
		\draw[very thick,red] (b4) -- (b5);
		\draw[very thick,red] (a4) -- (a5);
		\draw[very thick,red] (c4) -- (c5);

		\end{tikzpicture}
		\caption{A triangulation $ T_n $ on $ n = 6k+1 $ vertices which is conjectured to maximize the proximity  among triangulations of this order.}
		\label{fig:t1}
	
\end{figure}

\begin{figure}[htbp] 
	\centering
		\begin{tikzpicture}
		[scale=0.8,xscale=1,yscale=0.9,rotate=180,inner sep=1mm, %
		vertex/.style={circle,thick,draw}, %
		thickedge/.style={line width=2pt}] %
		\node[vertex] (a1) at (0,4) [fill=white] {};
		\node[vertex] (a2) at (2,4) [fill=white] {};
		\node[vertex,red] (a3) at (4,4) [fill=white] {};
		\node[vertex,red] (a4) at (6,4) [fill=white] {};
		\node[vertex] (a5) at (8,4) [fill=white] {};
		\node[vertex] (a6) at (10,4) [fill=white] {};
		\node[vertex] (a0) at (-2,4) [fill=white] {};

		\node[vertex] (b1) at (0,0) [fill=white] {};
		\node[vertex] (b2) at (2,0) [fill=white] {};
		\node[vertex,red] (b3) at (4,0) [fill=white] {};
		\node[vertex,red] (b4) at (6,0) [fill=white] {};
		\node[vertex] (b5) at (8,0) [fill=white] {};
		\node[vertex] (b6) at (10,0) [fill=white] {};
		
		\node[vertex] (c1) at (0.5,2) [fill=white] {};
		\node[vertex] (c2) at (2.5,2) [fill=white] {};
		\node[vertex,red] (c3) at (4.5,2) [fill=black] {};
		\node[vertex,red] (c4) at (6.5,2) [fill=white] {};
		\node[vertex] (c5) at (8.5,2) [fill=white] {};
		\node[vertex] (c6) at (10.5,2) [fill=white] {};
		\node[vertex] (c0) at (-1.5,2) [fill=white] {};
		
		\draw[very thick] (a0)--(a1)--(a2)--(a3) (a5)--(a6);  
		\draw[very thick] (b1)--(b2)--(b3) (b5)--(b6);     
		\draw[very thick] (c0)--(c1)--(c2)--(c3) (c5)--(c6); 
		\draw[very thick] (a1)--(b1)--(c1)--(a1);
		\draw[very thick] (a2)--(b2)--(c2)--(a2);
		\draw[very thick,red] (a3)--(b3)--(c3)--(a3);
		\draw[very thick,red] (a4)--(b4)--(c4)--(a4);
		\draw[very thick] (a5)--(b5)--(c5)--(a5);
		\draw[very thick] (a6)--(b6)--(c6)--(a6);
		
		\draw[very thick] (c1) -- (a2);  
		\draw[very thick] (c1) -- (b2);  		
		\draw[very thick] (c2) -- (a3);  
		\draw[very thick] (c2) -- (b3);
		\draw[very thick,red] (c3) -- (a4);  
		\draw[very thick,red] (c3) -- (b4);
		\draw[very thick,red] (c4) -- (a5);  
		\draw[very thick,red] (c4) -- (b5);
		
		\draw[very thick] (a1) -- (b2);  
		\draw[very thick] (a2) -- (b3);
		\draw[very thick,red] (a3) -- (b4);
		\draw[very thick,red] (a4) -- (b5);
		
		\draw[very thick,red] (a3) -- (a4);
		\draw[very thick,red] (b3) -- (b4);
		\draw[very thick,red] (c3) -- (c4);
		\draw[very thick,red] (c3) -- (c4);
		
		\draw[very thick] (c5) -- (a6);
		\draw[very thick] (c5) -- (b6);
		\draw[very thick] (a5) -- (b6);
		
		\draw[very thick] (c0) -- (a1);
		\draw[very thick] (c0) -- (b1);
		
		\draw[very thick] (a0) -- (b1);
		\draw[very thick] (a0) -- (c0);
		
		\draw[very thick,red] (b4) -- (b5);
		\draw[very thick,red] (a4) -- (a5);
		\draw[very thick,red] (c4) -- (c5);
		
		\end{tikzpicture}
		\caption{A triangulation $ T_n $ on $ n = 6k+2 $ vertices which is conjectured to maximize the proximity among triangulations of this order.}
		\label{fig:t2}
	
\end{figure}

\begin{figure}[htbp] 
	\centering
		\begin{tikzpicture}
		[scale=0.8,xscale=1,yscale=0.9,rotate=180,inner sep=1mm, %
		vertex/.style={circle,thick,draw}, %
		thickedge/.style={line width=2pt}] %
		\node[vertex] (a1) at (0,4) [fill=white] {};
		\node[vertex] (a2) at (2,4) [fill=white] {};
		\node[vertex] (a3) at (4,4) [fill=white] {};
		\node[vertex,red] (a4) at (6,4) [fill=white] {};
		\node[vertex,red] (a5) at (8,4) [fill=white] {};
		\node[vertex] (a6) at (10,4) [fill=white] {};
		
		\node[vertex] (b1) at (0,0) [fill=white] {};
		\node[vertex] (b2) at (2,0) [fill=white] {};
		\node[vertex] (b3) at (4,0) [fill=white] {};
		\node[vertex,red] (b4) at (6,0) [fill=white] {};
		\node[vertex,red] (b5) at (8,0) [fill=white] {};
		\node[vertex] (b6) at (10,0) [fill=white] {};
		
		\node[vertex] (c1) at (0.5,2) [fill=white] {};
		\node[vertex] (c2) at (2.5,2) [fill=white] {};
		\node[vertex] (c3) at (4.5,2) [fill=white] {};
		\node[vertex,red] (c4) at (6.5,2) [fill=black] {};
		\node[vertex,red] (c5) at (8.5,2) [fill=white] {};
		\node[vertex] (c6) at (10.5,2) [fill=white] {};
		
		\draw[very thick] (a1)--(a2)--(a3) (a6);  
		\draw[very thick] (b1)--(b2)--(b3) (b6);     
		\draw[very thick] (c1)--(c2)--(c3) (c6); 
		\draw[very thick] (a1)--(b1)--(c1)--(a1);
		\draw[very thick] (a2)--(b2)--(c2)--(a2);
		\draw[very thick] (a3)--(b3)--(c3)--(a3);
		\draw[very thick,red] (a4)--(b4)--(c4)--(a4);
		\draw[very thick,red] (a5)--(b5)--(c5)--(a5);
		\draw[very thick] (a6)--(b6)--(c6)--(a6);
		
		\draw[very thick] (c1) -- (a2);  
		\draw[very thick] (c1) -- (b2);  		
		\draw[very thick] (c2) -- (a3);  
		\draw[very thick] (c2) -- (b3);
		\draw[very thick] (c3) -- (a4);  
		\draw[very thick] (c3) -- (b4);
		\draw[very thick,red] (c4) -- (a5);  
		\draw[very thick,red] (c4) -- (b5);
		
		\draw[very thick] (a1) -- (b2);  
		\draw[very thick] (a2) -- (b3);
		\draw[very thick] (a3) -- (b4);
		\draw[very thick,red] (a4) -- (b5);
		
		\draw[very thick] (a3) -- (a4);
		\draw[very thick] (b3) -- (b4);
		\draw[very thick] (c3) -- (c4);
		\draw[very thick] (c3) -- (c4);
		
		\draw[very thick,red] (c5) -- (a6);
		\draw[very thick,red] (c5) -- (b6);
		\draw[very thick,red] (a5) -- (b6);

		\node[vertex] (a7) at (12,4) [fill=white] {};
		\node[vertex] (b7) at (12,0) [fill=white] {};
		\node[vertex] (c7) at (12.5,2) [fill=white] {};
		\draw[very thick] (a6) -- (a7);
		\draw[very thick] (b6) -- (b7);
		\draw[very thick] (c6) -- (c7);
		\draw[very thick] (a7)--(b7)--(c7)--(a7);
		\draw[very thick] (c6)--(a7);
		\draw[very thick] (c6)--(b7);
		\draw[very thick] (a6)--(b7);
		
		\draw[very thick,red] (b6) -- (b5);
		\draw[very thick,red] (a6) -- (a5);
		\draw[very thick,red] (c6) -- (c5);
		
		\draw[very thick,red] (b4) -- (b5);
		\draw[very thick,red] (a4) -- (a5);
		\draw[very thick,red] (c4) -- (c5);
		
		\end{tikzpicture}
		\caption{A triangulation $ T_n $ on $ n = 6k+3 $ vertices which is conjectured to maximize the proximity among triangulations of this order.}
		\label{fig:t3}
	
\end{figure}

\begin{figure}[htbp] 
	\centering
		\begin{tikzpicture}
		[scale=0.8,xscale=1,yscale=0.9,rotate=180,inner sep=1mm, %
		vertex/.style={circle,thick,draw}, %
		thickedge/.style={line width=2pt}] %
		\node[vertex] (a1) at (0,4) [fill=white] {};
		\node[vertex] (a2) at (2,4) [fill=white] {};
		\node[vertex] (a3) at (4,4) [fill=white] {};
		\node[vertex,red] (a4) at (6,4) [fill=white] {};
		\node[vertex,red] (a5) at (8,4) [fill=white] {};
		\node[vertex] (a6) at (10,4) [fill=white] {};
		
		\node[vertex] (b1) at (0,0) [fill=white] {};
		\node[vertex] (b2) at (2,0) [fill=white] {};
		\node[vertex] (b3) at (4,0) [fill=white] {};
		\node[vertex,red] (b4) at (6,0) [fill=white] {};
		\node[vertex,red] (b5) at (8,0) [fill=white] {};
		\node[vertex] (b6) at (10,0) [fill=white] {};
		
		\node[vertex] (c1) at (0.5,2) [fill=white] {};
		\node[vertex] (c2) at (2.5,2) [fill=white] {};
		\node[vertex] (c3) at (4.5,2) [fill=white] {};
		\node[vertex,red] (c4) at (6.5,2) [fill=black] {};
		\node[vertex,red] (c5) at (8.5,2) [fill=white] {};
		\node[vertex] (c6) at (10.5,2) [fill=white] {};
		\node[vertex] (c0) at (-1.5,2) [fill=white] {};
		
		\draw[very thick] (a1)--(a2)--(a3) (a6);  
		\draw[very thick] (b1)--(b2)--(b3) (b6);     
		\draw[very thick] (c0)--(c1)--(c2)--(c3) (c6); 
		\draw[very thick] (a1)--(b1)--(c1)--(a1);
		\draw[very thick] (a2)--(b2)--(c2)--(a2);
		\draw[very thick] (a3)--(b3)--(c3)--(a3);
		\draw[very thick,red] (a4)--(b4)--(c4)--(a4);
		\draw[very thick,red] (a5)--(b5)--(c5)--(a5);
		\draw[very thick] (a6)--(b6)--(c6)--(a6);
		
		\draw[very thick] (c1) -- (a2);  
		\draw[very thick] (c1) -- (b2);  		
		\draw[very thick] (c2) -- (a3);  
		\draw[very thick] (c2) -- (b3);
		\draw[very thick] (c3) -- (a4);  
		\draw[very thick] (c3) -- (b4);
		\draw[very thick,red] (c4) -- (a5);  
		\draw[very thick,red] (c4) -- (b5);
		
		\draw[very thick] (a1) -- (b2);  
		\draw[very thick] (a2) -- (b3);
		\draw[very thick] (a3) -- (b4);
		\draw[very thick,red] (a4) -- (b5);
		
		\draw[very thick] (a3) -- (a4);
		\draw[very thick] (b3) -- (b4);
		\draw[very thick] (c3) -- (c4);
		\draw[very thick] (c3) -- (c4);
		
		\draw[very thick,red] (c5) -- (a6);
		\draw[very thick,red] (c5) -- (b6);
		\draw[very thick,red] (a5) -- (b6);
		
		\draw[very thick] (c0) -- (a1);
		\draw[very thick] (c0) -- (b1);
		
		\node[vertex] (a7) at (12,4) [fill=white] {};
		\node[vertex] (b7) at (12,0) [fill=white] {};
		\node[vertex] (c7) at (12.5,2) [fill=white] {};
		\draw[very thick] (a6) -- (a7);
		\draw[very thick] (b6) -- (b7);
		\draw[very thick] (c6) -- (c7);
		\draw[very thick] (a7)--(b7)--(c7)--(a7);
		\draw[very thick] (c6)--(a7);
		\draw[very thick] (c6)--(b7);
		\draw[very thick] (a6)--(b7);
		
		\draw[very thick,red] (b6) -- (b5);
		\draw[very thick,red] (a6) -- (a5);
		\draw[very thick,red] (c6) -- (c5);
		
		\draw[very thick,red] (b4) -- (b5);
		\draw[very thick,red] (a4) -- (a5);
		\draw[very thick,red] (c4) -- (c5);
		
		\end{tikzpicture}
		\caption{A triangulation $ T_n $ on $ n = 6k+4 $ vertices which is conjectured to maximize the proximity among triangulations of this order.}
		\label{fig:t4}
	
\end{figure}

\begin{figure}[htbp] 
	\centering
	\begin{tikzpicture}
	[scale=0.8,xscale=1,yscale=0.9,rotate=180,inner sep=1mm, %
	vertex/.style={circle,thick,draw}, %
	thickedge/.style={line width=2pt}] %
	\node[vertex] (a1) at (0,4) [fill=white] {};
	\node[vertex] (a2) at (2,4) [fill=white] {};
	\node[vertex] (a3) at (4,4) [fill=white] {};
	\node[vertex,red] (a4) at (6,4) [fill=white] {};
	\node[vertex,red] (a5) at (8,4) [fill=white] {};
	\node[vertex] (a6) at (10,4) [fill=white] {};
	\node[vertex] (a0) at (-2,4) [fill=white] {};

	\node[vertex] (b1) at (0,0) [fill=white] {};
	\node[vertex] (b2) at (2,0) [fill=white] {};
	\node[vertex] (b3) at (4,0) [fill=white] {};
	\node[vertex,red] (b4) at (6,0) [fill=white] {};
	\node[vertex,red] (b5) at (8,0) [fill=white] {};
	\node[vertex] (b6) at (10,0) [fill=white] {};
	
	\node[vertex] (c1) at (0.5,2) [fill=white] {};
	\node[vertex] (c2) at (2.5,2) [fill=white] {};
	\node[vertex] (c3) at (4.5,2) [fill=black] {};
	\node[vertex,red] (c4) at (6.5,2) [fill=white] {};
	\node[vertex,red] (c5) at (8.5,2) [fill=white] {};
	\node[vertex] (c6) at (10.5,2) [fill=white] {};
	\node[vertex] (c0) at (-1.5,2) [fill=white] {};
	
	\draw[very thick] (a0)--(a1)--(a2)--(a3) (a6);  
	\draw[very thick] (b1)--(b2)--(b3) (b6);     
	\draw[very thick] (c0)--(c1)--(c2)--(c3) (c6); 
	\draw[very thick] (a1)--(b1)--(c1)--(a1);
	\draw[very thick] (a2)--(b2)--(c2)--(a2);
	\draw[very thick] (a3)--(b3)--(c3)--(a3);
	\draw[very thick,red] (a4)--(b4)--(c4)--(a4);
	\draw[very thick,red] (a5)--(b5)--(c5)--(a5);
	\draw[very thick] (a6)--(b6)--(c6)--(a6);
	
	\draw[very thick] (c1) -- (a2);  
	\draw[very thick] (c1) -- (b2);  		
	\draw[very thick] (c2) -- (a3);  
	\draw[very thick] (c2) -- (b3);
	\draw[very thick] (c3) -- (a4);  
	\draw[very thick] (c3) -- (b4);
	\draw[very thick,red] (c4) -- (a5);  
	\draw[very thick,red] (c4) -- (b5);
	
	\draw[very thick] (a1) -- (b2);  
	\draw[very thick] (a2) -- (b3);
	\draw[very thick] (a3) -- (b4);
	\draw[very thick,red] (a4) -- (b5);
	
	\draw[very thick] (a3) -- (a4);
	\draw[very thick] (b3) -- (b4);
	\draw[very thick] (c3) -- (c4);
	\draw[very thick] (c3) -- (c4);
	
	\draw[very thick,red] (c5) -- (a6);
	\draw[very thick,red] (c5) -- (b6);
	\draw[very thick,red] (a5) -- (b6);
	
	\draw[very thick] (c0) -- (a1);
	\draw[very thick] (c0) -- (b1);
	
	\draw[very thick] (a0) -- (b1);
	\draw[very thick] (a0) -- (c0);
	
	\node[vertex] (a7) at (12,4) [fill=white] {};
	\node[vertex] (b7) at (12,0) [fill=white] {};
	\node[vertex] (c7) at (12.5,2) [fill=white] {};
	\draw[very thick] (a6) -- (a7);
	\draw[very thick] (b6) -- (b7);
	\draw[very thick] (c6) -- (c7);
	\draw[very thick] (a7)--(b7)--(c7)--(a7);
	\draw[very thick] (c6)--(a7);
	\draw[very thick] (c6)--(b7);
	\draw[very thick] (a6)--(b7);
	
	\draw[very thick,red] (b6) -- (b5);
	\draw[very thick,red] (a6) -- (a5);
	\draw[very thick,red] (c6) -- (c5);
	
	\draw[very thick,red] (b4) -- (b5);
	\draw[very thick,red] (a4) -- (a5);
	\draw[very thick,red] (c4) -- (c5);
	
	\end{tikzpicture}
	\caption{A triangulation $ T_n $ on $ n = 6k+5 $ vertices which is conjectured to maximize the proximity among triangulations of this order.}
	\label{fig:t5}
	
\end{figure}

\begin{equation} \label{tpoly}
\pi(T_n) =
\begin{cases}
\frac{n+5}{12} + \frac{5}{12(n-1)} & \text{if } n=6k \\
\frac{n+5}{12} & \text{if } n=6k+1 \\
\frac{n+5}{12} + \frac{5}{12(n-1)} & \text{if } n=6k+2 \\
\frac{n+5}{12} + \frac{2}{3(n-1)} & \text{if }n=6k+3 \\
\frac{n+5}{12} + \frac{3}{4(n-1)} & \text{if } n=6k+4 \\
\frac{n+5}{12} + \frac{2}{3(n-1)} & \text{if } n=6k+5 \\
\end{cases}
\end{equation}

\begin{table}
\begin{tabular}{|c|c|c|} 
			\hline
			Order & Min Status & Count  \\ [0.5ex] 
			\hline
			4 & 3 & 1 \\ 
			5 & 4 & 1 \\ 
			6 & 6 & 1 \\ 
			7 & 7 & 2 \\ 
			8 & 9 & 2 \\ 
			9 & 11 & 1 \\ 
			10 & 13 & 1 \\ 
			11 & 14 & 44 \\ 
			12 & 18 & 1 \\ 
			13 & 19 & 2 \\ 
			14 & 22 & 1 \\ 
			15 & 24 & 2 \\ 
			16 & 27 & 2 \\ 
			17 & 30 & 1 \\ 
			18 & 33 & 3 \\ 
			\hline
		\end{tabular}
			\label{tab:tsummary}
	\centering
	\caption{A summary of the largest minimum status among all triangulations.}	
 \end{table}

\begin{figure}[htbp]
	\centering
		%
		\caption{A 4-connected triangulation $ T_n^4 $ on $ n = 8k+2 $ vertices which is conjectured to maximize the proximity among triangulations of this order and connectivity.}
		\label{fig:t4c2}
	
\end{figure}

\begin{figure}[htbp]
	\centering
		%
		\caption{A 4-connected triangulation $ T_n^4 $ on $ n = 8k+3 $ vertices which is conjectured to maximize the proximity among triangulations of this order and connectivity.}
		\label{fig:t4c3}
	
\end{figure}

\begin{figure}[htbp]
	\centering
		%
		\caption{A 4-connected triangulation $ T_n^4 $ on $ n = 8k+4 $ vertices which is conjectured to maximize the proximity among triangulations of this order and connectivity.}
		\label{fig:t4c4}
	
\end{figure}

\begin{figure}[htbp]
	\centering
		%
		\caption{A 4-connected triangulation $ T_n^4 $ on $ n = 8k $ vertices which is conjectured to maximize the proximity among triangulations of this order and connectivity.}
		\label{fig:t4c0}
	
\end{figure}

\begin{figure}[htbp]
	\centering
		%
		\caption{A 4-connected triangulation $ T_n^4 $ on $ n = 8k+1 $ vertices which is conjectured to maximize the proximity among triangulations of this order and connectivity.}
		\label{fig:t4c1}
	
\end{figure}

\begin{equation} \label{t4cpoly}
\pi(T_n^4) =
\begin{cases}
\frac{n+9}{16} + \frac{21}{16(n-1)} & \text{if } n=8k+2 \\
\frac{n+9}{16} + \frac{3}{2(n-1)} & \text{if } n=8k+3 \\
\frac{n+9}{16} + \frac{25}{16(n-1)} & \text{if } n=8k+4 \\
\frac{n+9}{16} + \frac{3}{2(n-1)} & \text{if } n=8k+5 \\
\frac{n+9}{16} + \frac{21}{16(n-1)} & \text{if } n=8k+6 \\
\frac{n+9}{16} + \frac{1}{n-1} & \text{if } n=8k+7 \\
\frac{n+9}{16} + \frac{25}{16(n-1)} & \text{if } n=8k \\
\frac{n+9}{16} + \frac{1}{n-1} & \text{if } n=8k+1 \\
\end{cases}
\end{equation}

\begin{table}
\begin{tabular}{ |c|c|c| } 
			\hline
			Order & Min Status & Count  \\ [0.5ex] 
			\hline
			6 & 6 & 1 \\ 
			7 & 7 & 1 \\ 
			8 & 9 & 1 \\ 
			9 & 11 & 1 \\ 
			10 & 13 & 1 \\ 
			11 & 14 & 10 \\ 
			12 & 18 & 1 \\ 
			13 & 19 & 1 \\ 
			14 & 21 & 13 \\ 
			15 & 24 & 1 \\ 
			16 & 27 & 1 \\ 
			17 & 29 & 5 \\ 
			18 & 32 & 2 \\ 
			19 & 34 & 28 \\ 
			20 & 37 & 13 \\ 
			21 & 40 & 6 \\ 
			22 & 44 & 5 \\ 
			\hline
		\end{tabular}
		\label{tab:t4csummary}
	\centering
	\caption{A summary of the largest minimum status among all 4-connected triangulations.}		
\end{table}

\begin{figure}[htbp] 
	\centering
		%
		\caption{A 5-connected triangulation $ T_n^5 $ on $ n = 10k+7 $ vertices which is conjectured to maximize the proximity among triangulations of this order and connectivity.}
		\label{fig:t5c7}
	
\end{figure}

\begin{figure}[htbp] 
	\centering
		%
		\caption{A 5-connected triangulation $ T_n^5 $ on $ n = 10k+8 $ vertices which is conjectured to maximize the proximity among triangulations of this order and connectivity. Here, the orange pattern is to be repeated twice rather than once.}
		\label{fig:t5c8}
	
\end{figure}

\begin{figure}[htbp] 
	\centering
		%
		\caption{A 5-connected triangulation $ T_n^5 $ on $ n = 10k+9 $ vertices which is conjectured to maximize the proximity among triangulations of this order and connectivity.}
		\label{fig:t5c9}
	
\end{figure}

\begin{figure}[htbp] 
	\centering
		%
		\caption{A 5-connected triangulation $ T_n^5 $ on $ n = 10k $ vertices which is conjectured to maximize the proximity among triangulations of this order and connectivity.}
		\label{fig:t5c0}
	
\end{figure}

\begin{figure}[htbp] 
	\centering
		%
		\caption{A 5-connected triangulation $ T_n^5 $ on $ n = 10k+1 $ vertices which is conjectured to maximize the proximity among triangulations of this order and connectivity.}
		\label{fig:t5c1}
	
\end{figure}

\begin{figure}[htbp] 
	\centering
		%
		\caption{A 5-connected triangulation $ T_n^5 $ on $ n = 10k+2 $ vertices which is conjectured to maximize the proximity among triangulations of this order and connectivity. }
		\label{fig:t5c2}
	
\end{figure}

\begin{figure}[htbp] 
	\centering
		%
		\caption{A 5-connected triangulation $ T_n^5 $ on $ n = 10k+6 $ vertices which is conjectured to maximize the proximity among triangulations of this order and connectivity.}
		\label{fig:t5c6}
	
\end{figure}

\begin{table}
		\begin{tabular}{ |c|c|c| } 
			\hline
			Order & Min Status & Count  \\ [0.5ex] 
			\hline
			12 & 18 & 1 \\ 
			13 & --- & 0 \\ 
			14 & 21 & 1 \\ 
			15 & 24 & 1 \\ 
			16 & 27 & 1 \\ 
			17 & 29 & 1 \\ 
			18 & 32 & 1 \\ 
			19 & 34 & 4 \\ 
			20 & 37 & 6 \\ 
			21 & 40 & 3 \\ 
			22 & 44 & 2 \\ 
			23 & 46 & 5 \\ 
			24 & 49 & 19 \\ 
			25 & 52 & 18 \\ 
			26 & 56 & 3 \\ 
			27 & 60 & 3 \\ 
			28 & 63 & 3 \\ 
			29 & 66 & 2 \\ 
			30 & 69 & 59 \\ 
			31 & 73 & 2 \\ 
			32 & 80 & 1 \\ 	
			\hline
		\end{tabular}
	\centering
	\caption{A summary of the largest minimum status among all 5-connected triangulations.}
	\label{tab:t5csummary}
\end{table}

\begin{equation}
\pi(T_n^5) =
\begin{cases} \label{t5cpoly}
\frac{n+13}{20} & \text{if } n=10k+7 \\
\frac{n+13}{20} - \frac{7}{20(n-1)} & \text{if } n=10k+8 \\
\frac{n+13}{20} - \frac{4}{5(n-1)} & \text{if } n=10k+9 \\
\frac{n+13}{20} - \frac{7}{20(n-1)} & \text{if } n=10k \\
\frac{n+13}{20} - \frac{1}{n-1} & \text{if } n=10k+1 \\
\frac{n+13}{20} - \frac{1}{4(n-1)} & \text{if } n=10k+2 \\
\frac{n+13}{20} - \frac{8}{5(n-1)} & \text{if } n=10k+3 \\
\frac{n+13}{20} - \frac{11}{20(n-1)} & \text{if } n=10k+4 \\
\frac{n+13}{20} - \frac{3}{5(n-1)} & \text{if } n=10k+5 \\
\frac{n+13}{20} - \frac{3}{4(n-1)} & \text{if } n=10k+6 \\
\end{cases}
\end{equation}

\begin{figure}[htbp]
	\centering

		\caption{A quadrangulation $ Q_n $ on $ n = 4k $ vertices which is conjectured to maximize the proximity among quadrangulations of this order.}
		\label{fig:q0}

\end{figure}
 
\begin{figure}[htbp]
	\centering
		%
		\caption{A quadrangulation $ Q_n $ on $ n = 4k+1 $ vertices which is conjectured to maximize the proximity among quadrangulations of this order.}
		\label{fig:q1}

\end{figure}
 
\begin{figure}[htbp]
	\centering
		%
		\caption{A quadrangulation $ Q_n $ on $ n = 4k+2 $ vertices which is conjectured to maximize the proximity among quadrangulations of this order.}
		\label{fig:q2}

\end{figure}
 
\begin{figure}[htbp]
	\centering
		%
		\caption{A quadrangulation $ Q_n $ on $ n = 4k+3 $ vertices which is conjectured to maximize the proximity among quadrangulations of this order.}
		\label{fig:q3}

\end{figure}

\begin{equation} \label{qpoly}
\pi(Q_n) =
\begin{cases}
\frac{n+1}{8} + \frac{17}{8(n-1)} & \text{if } n=4k \\
\frac{n+1}{8} + \frac{2}{n-1} & \text{if } n=4k+1 \\
\frac{n+1}{8} + \frac{21}{8(n-1)} & \text{if } n=4k+2 \\
\frac{n+1}{8} + \frac{2}{n-1} & \text{if } n=4k+3 \\
\end{cases}
\end{equation}

\begin{table}
		\begin{tabular}{ |c|c|c| } 
			\hline
			Order & Min Status & Count  \\ [0.5ex] 
			\hline
			4 & 4 & 1 \\ 
			5 & 5 & 1 \\ 
			6 & 7 & 1 \\ 
			7 & 8 & 2 \\ 
			8 & 12 & 1 \\ 
			9 & 13 & 1 \\ 
			10 & 16 & 1 \\ 
			11 & 18 & 1 \\ 
			12 & 20 & 19 \\ 
			13 & 23 & 1 \\ 
			14 & 28 & 1 \\ 
			15 & 30 & 1 \\ 
			16 & 34 & 2 \\ 
			17 & 38 & 1 \\ 
			18 & 43 & 1 \\ 
			19 & 47 & 1 \\ 
			20 & 52 & 2 \\ 
			\hline
		\end{tabular}
			\label{tab:qsummary}
	\centering
	\caption{A summary of the largest minimum status among all quadrangulations.}
\end{table}

\begin{figure}[htbp] 
	\centering
		%
		\caption{A 3-connected quadrangulation $ Q_n^3 $ on $ n = 6k+2 $ vertices which is conjectured to maximize the proximity among quadrangulations of this order and connectivity.}
		\label{fig:q3c2} 
	
\end{figure}
 
\begin{figure}[htbp] 
	\centering
		%
		\caption{A 3-connected quadrangulation $ Q_n^3 $ on $ n = 6k+3 $ vertices which is conjectured to maximize the proximity among quadrangulations of this order and connectivity.}
		\label{fig:q3c3}
	
\end{figure}

\begin{figure}[htbp] 
	\centering
		%
		\caption{A 3-connected quadrangulation $ Q_n^3 $ on $ n = 6k+4 $ vertices which is conjectured to maximize the proximity among quadrangulations of this order and connectivity.}
		\label{fig:q3c4} 
	
\end{figure}
 
\begin{figure}[htbp]
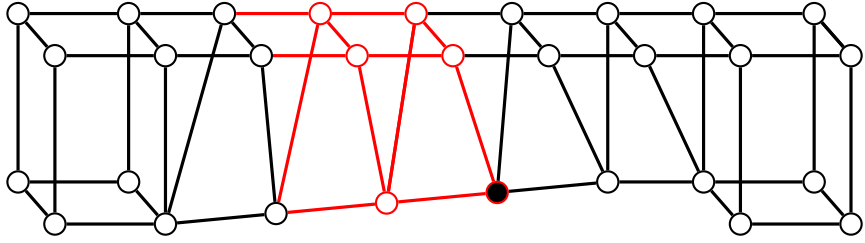
 
	\centering
		%
		\caption{A 3-connected quadrangulation $ Q_n^3 $ on $ n = 6k $ vertices which is conjectured to maximize the proximity.}
		\label{fig:q3c0}
	
\end{figure}
 
\begin{figure}[htbp] 
	\centering
		%
		\caption{A 3-connected quadrangulation $ Q_n^3 $ on $ n = 6k+1 $ vertices which is conjectured to maximize the proximity among quadrangulations of this order and connectivity.}
		\label{fig:q3c1} 
	
\end{figure}

	\begin{equation}\label{q3cpoly}
	\pi(Q_n^3) =
	\begin{cases}
	\frac{n+9}{12} - \frac{23}{12(n-1)} & \text{if } n=6k+2 \\
	\frac{n+9}{12} - \frac{4}{n-1} & \text{if } n=6k+3 \\
	\frac{n+9}{12} - \frac{13}{4(n-1)} & \text{if } n=6k+4 \\
	\frac{n+9}{12} - \frac{8}{3(n-1)} & \text{if } n=6k+5 \\
	\frac{n+9}{12} - \frac{13}{4(n-1)} & \text{if } n=6k \\
	\frac{n+9}{12} - \frac{3}{n-1} & \text{if } n=6k+1 \\
	\end{cases}
	\end{equation}

\begin{table}
	\begin{tabular}{ |c|c|c| } 
		\hline
		Order & Min Status & Count \\ [0.5ex] 
		\hline
		
		8 & 12 & 1 \\ 
		9 & --- & 0 \\ 
		10 & 15 & 1 \\ 
		11 & 18 & 1 \\ 
		12 & 20 & 2 \\ 
		13 & 22 & 1 \\ 
		14 & 28 & 1 \\ 
		15 & 29 & 1 \\ 
		16 & 32 & 4 \\ 
		17 & 35 & 1 \\ 
		18 & 39 & 1 \\ 
		19 & 41 & 3 \\ 
		20 & 44 & 23 \\ 
		21 & 47 & 7 \\ 
		22 & 55 & 1 \\ 
		23 & 57 & 1 \\ 
		24 & 60 & 16 \\ 
		25 & 65 & 1 \\ 
		26 & 71 & 1 \\ 
		27 & 74 & 3 \\ 
		28 & 80 & 2 \\ 
		\hline
	\end{tabular}
	\label{tab:q3csummary}
	\centering
	\caption{A summary of the largest minimum status among all 3-connected quadrangulations.}
\end{table}